\documentclass{amsart}
\usepackage{amssymb}
\usepackage{amsfonts}

\newtheorem{theorem}{Theorem}
\theoremstyle{plain}

\newtheorem{definition}{Definition}

\numberwithin{equation}{section}
\input{tcilatex}

\begin{document}
\title[$L$-Lipschitzian functions ]{On the Hadamard's type inequalities for $%
L$-Lipschitzian mapping}
\author{Mehmet Zeki Sar\i kaya}
\address{Department of Mathematics, Faculty of Science and Arts, D\"{u}zce
University, D\"{u}zce, Turkey}
\email{sarikayamz@gmail.com}
\author{Hatice YALDIZ}
\email{yaldizhatice@gmail.com}
\date{}

\begin{abstract}
In this paper, we establish some new inequalities of Hadamard's type for $L$%
-Lipschitzian mapping in two variables.
\end{abstract}

\maketitle

\section{Introduction}

Let $f:I\subseteq \mathbb{R\rightarrow R}$ be a convex mapping defined on
the interval $I$ of real numbers and $a,b\in I$, with $a<b.$ the following
double inequality is well known in the literature as the Hermite-Hadamard
inequality:%
\begin{equation*}
f\left( \frac{a+b}{2}\right) \leq \frac{1}{b-a}\int_{a}^{b}f\left( x\right)
dx\leq \frac{f\left( a\right) +f\left( b\right) }{2}.
\end{equation*}

Let us now consider a bidemensional interval $\Delta =:\left[ a,b\right]
\times \left[ c,d\right] $ in $\mathbb{R}^{2}$ with $a<b$ and $c<d$. A
mapping $f:\Delta \rightarrow \mathbb{R}$ is said to be convex on $\Delta $
if the following inequality:%
\begin{equation*}
f(tx+\left( 1-t\right) z,ty+\left( 1-t\right) w)\leq tf\left( x,y\right)
+\left( 1-t\right) f\left( z,w\right)
\end{equation*}%
holds, for all $\left( x,y\right) ,\left( z,w\right) \in \Delta $ and $t\in %
\left[ 0,1\right] .$A function $f:\Delta \rightarrow \mathbb{R}$ is said to
be on the co-ordinates on $\Delta $ if the partial mappings $f_{y}:\left[ a,b%
\right] \rightarrow \mathbb{R},$ \ $f_{y}\left( u\right) =f\left( u,y\right) 
$ and $f_{x}:\left[ c,d\right] \rightarrow \mathbb{R},$ \ $f_{x}\left(
v\right) =f\left( x,v\right) $ are convex where defined for all $x\in \left[
a,b\right] $ and $y\in \left[ c,d\right] \ $(see \cite{D}).

A formal definition for co-ordinated convex function may be stated as
follows:

\begin{definition}
A function $f:\Delta \rightarrow \mathbb{R}$ will be called co-ordinated
canvex on $\Delta $, for all $t,s\in \lbrack 0,1]$ and $(x,y),(u,w)\in
\Delta ,$if the following inequality holds:%
\begin{eqnarray}
&&f(tx+\left( 1-t\right) y,su+\left( 1-s\right) w)  \notag \\
&&  \notag \\
&\leq &tsf(x,u)+s(1-t)f(y,u)+t(1-s)f(x,w)+(1-t)(1-s)f(y,w).  \label{a}
\end{eqnarray}
\end{definition}

Clearly, every convex function is co-ordinated convex. Furthermore, there
exist co-ordinated convex function which is not convex, (see, \cite{D}). For
several recent results concerning Hermite-Hadamard's inequality for some
convex function on the co-ordinates on a rectangle from the plane $\mathbb{R}%
^{2},$ we refer the reader to (\cite{AD1}-\cite{D}, \cite{MAL1}, \cite{MAL2}%
, \cite{MEO}, \cite{Qi} and \cite{wu}).

Recently, in \cite{D}, Dragomir establish the following similar inequality
of Hadamard's type for co-ordinated convex mapping on a rectangle from the
plane $\mathbb{R}^{2}.$

\begin{theorem}
\label{t1} Suppose that $f:\Delta \rightarrow \mathbb{R}$ is co-ordinated
convex on $\Delta .\ $Then one has the inequalities:%
\begin{eqnarray*}
&&f\left( \dfrac{a+b}{2},\dfrac{c+d}{2}\right) \\
&\leq &\dfrac{1}{2}\left[ \dfrac{1}{b-a}\dint_{a}^{b}f\left( x,\dfrac{c+d}{2}%
\right) dx+\dfrac{1}{d-c}\dint_{c}^{d}f\left( \dfrac{a+b}{2},y\right) dy%
\right] \\
&\leq &\dfrac{1}{\left( b-a\right) \left( d-c\right) }\dint_{a}^{b}%
\dint_{c}^{d}f\left( x,y\right) dydx \\
&\leq &\dfrac{1}{4}\left[ \dfrac{1}{b-a}\dint_{a}^{b}f\left( x,c\right) dx+%
\dfrac{1}{b-a}\dint_{a}^{b}f\left( x,d\right) dx\right. \\
&&\left. +\dfrac{1}{d-c}\dint_{c}^{d}f\left( a,y\right) dy+\dfrac{1}{d-c}%
\dint_{c}^{d}f\left( b,y\right) dy\right] \\
&\leq &\dfrac{f\left( a,c\right) +f\left( a,d\right) +f\left( b,c\right)
+f\left( b,d\right) }{4}.
\end{eqnarray*}%
The above inequalities are sharp.
\end{theorem}

\begin{definition}
Consider a function $f:V\rightarrow \mathbb{R}$ defined on a subset $V$ of $%
\mathbb{R}^{n},\ n\in \mathbb{N}$. Let $L=(L_{1},L_{2},...,L_{n})$ where $%
L_{i}\geq 0,\ i=1,2,...,n.$ We say that $f$ is $L$-Lipschitzian function if 
\begin{equation*}
\left\vert f(x)-f(y)\right\vert \leq \dsum\limits_{i=1}^{n}L\left\vert
x_{i}-y_{i}\right\vert
\end{equation*}%
for all $x,y\in V.$
\end{definition}

The authors in \cite{D1} and \cite{yang}\ have proved the following
inequalities of Hadamard's type for Lipschitzian mapping.

\begin{theorem}
Let $f:I\subset \mathbb{R}\rightarrow \mathbb{R}$ be $L$-Lipschitzian on $I$
and $a,b\in I$ with $a<b$. Then, we have%
\begin{equation*}
\left\vert \frac{f(a)+f(b)}{2}-\frac{1}{b-a}\int_{a}^{b}f\left( x\right)
dx\right\vert \leq \frac{L(b-a)}{3}
\end{equation*}%
and 
\begin{equation*}
\left\vert f\left( \frac{a+b}{2}\right) -\frac{1}{b-a}\int_{a}^{b}f\left(
x\right) dx\right\vert \leq \frac{L(b-a)}{4}.
\end{equation*}
\end{theorem}

For several recent results concerning Hadamard's type inequality for some $L$%
-Lipschitzian function$,$ we refer the reader to (\cite{D1}, \cite{matic}, 
\cite{yang}).

The main purpose of this paper is to establish some Hadamard's type
ineqaulities for $L$-Lipschitzian mapping in two variables.

\section{Hadamard's Type Inequalities}

Firstly, we will start the proof of the Theorem \ref{t1} by using the
definition of the co-ordinated convex functions as follows:

\begin{theorem}
Suppose that $f:\Delta \rightarrow \mathbb{R}$ is co-ordinated convex on $%
\Delta .\ $Then one has the inequalities:%
\begin{eqnarray}
f\left( \dfrac{a+b}{2},\dfrac{c+d}{2}\right) &\leq &\dfrac{1}{\left(
b-a\right) \left( d-c\right) }\dint_{a}^{b}\dint_{c}^{d}f\left( x,y\right)
dydx  \notag \\
&&  \label{a1} \\
&\leq &\dfrac{f\left( a,c\right) +f\left( a,d\right) +f\left( b,c\right)
+f\left( b,d\right) }{4}.  \notag
\end{eqnarray}
\end{theorem}

\begin{proof}
According to (\ref{a}) with $x=t_{1}a+(1-t_{1})b$, $y=(1-t_{1})a+t_{1}b$, $%
u=s_{1}c+(1-s_{1})d,\ w=(1-s_{1})c+s_{1}d\ $and $t=s=\frac{1}{2},$ we find
that%
\begin{eqnarray*}
&&f\left( \dfrac{a+b}{2},\dfrac{c+d}{2}\right) \\
&& \\
&\leq &\frac{1}{4}\left[
f(t_{1}a+(1-t_{1})b,s_{1}c+(1-s_{1})d)+f((1-t_{1})a+t_{1}b,s_{1}c+(1-s_{1})d)\right.
\\
&& \\
&&\left.
+f(t_{1}a+(1-t_{1})b,(1-s_{1})c+s_{1}d)+f(t_{1}a+(1-t_{1})b,(1-s_{1})c+s_{1}d) 
\right] .
\end{eqnarray*}%
Thus, by integrating with respect to $t_{1},s_{1}$ on $[0,1]\times \lbrack
0,1]$, we obtain%
\begin{eqnarray*}
&&f\left( \dfrac{a+b}{2},\dfrac{c+d}{2}\right) \\
&& \\
&\leq &\frac{1}{4}\left[ \dint_{0}^{1}\dint_{0}^{1}\left[ f\left(
ta+(1-t)b,sc+(1-s)d\right) +f\left( ta+(1-t)b,(1-s)c+sd\right) \right]
dsdt\right. \\
&& \\
&&\left. +\dint_{0}^{1}\dint_{0}^{1}\left[ f\left(
(1-t)a+tb,sc+(1-s)d\right) +f\left( (1-t)a+tb,(1-s)c+sd\right) \right] dsdt%
\right] .
\end{eqnarray*}%
Using the change of the variable, we get%
\begin{equation}
f\left( \dfrac{a+b}{2},\dfrac{c+d}{2}\right) \leq \dfrac{1}{\left(
b-a\right) \left( d-c\right) }\dint_{a}^{b}\dint_{c}^{d}f\left( x,y\right)
dydx,  \label{a2}
\end{equation}%
which the first inequality is proved. The proof of the second inequality
follows by using (\ref{a}) with $x=a,$ $y=b,\ u=c$ and $w=d,$ and
integrating with respect to $t,s$ over $[0,1]\times \lbrack 0,1],$%
\begin{eqnarray*}
&&\dint_{0}^{1}\dint_{0}^{1}f\left( ta+(1-t)b,sc+(1-s)d\right) dsdt \\
&& \\
&\leq &\dint_{0}^{1}\dint_{0}^{1}\left[
tsf(a,c)+s(1-t)f(b,c)+t(1-s)f(a,d)+(1-t)(1-s)f(b,d)\right] dsdt \\
&& \\
&=&\dfrac{f\left( a,c\right) +f\left( a,d\right) +f\left( b,c\right)
+f\left( b,d\right) }{4}.
\end{eqnarray*}%
Here, using the change of the variable $x=ta+(1-t)b$ and $y=sc+(1-s)d$ for $%
s,t\in \lbrack 0,1],$ we have%
\begin{equation}
\dfrac{1}{\left( b-a\right) \left( d-c\right) }\dint_{a}^{b}\dint_{c}^{d}f%
\left( x,y\right) dydx\leq \dfrac{f\left( a,c\right) +f\left( a,d\right)
+f\left( b,c\right) +f\left( b,d\right) }{4}.  \label{a3}
\end{equation}%
We get the inequality (\ref{a1}) from (\ref{a2}) and (\ref{a3}). The proof
is complete.
\end{proof}

\begin{theorem}
\label{z} Let $f:\Delta \subset \mathbb{R}^{2}\rightarrow \mathbb{R}$
satisfy $L$-Lipschitzian conditions. That is, for $(t_{1},s_{1})$ and $%
(t_{2},s_{2})$ belong to $\Delta :=\left[ a,b\right] \times \left[ c,d\right]
,$ then we have%
\begin{equation*}
\left\vert f(t_{1},s_{1})-f(t_{2},s_{2})\right\vert \leq L_{1}\left\vert
t_{1}-t_{2}\right\vert +L_{2}\left\vert s_{1}-s_{2}\right\vert
\end{equation*}%
where $L_{1}$ and $L_{2}$ are positive constants. Then, we have the
following inequalities:%
\begin{equation}
\left\vert f\left( \dfrac{a+b}{2},\dfrac{c+d}{2}\right) -\dfrac{1}{\left(
b-a\right) \left( d-c\right) }\dint_{a}^{b}\dint_{c}^{d}f\left( x,y\right)
dydx\right\vert \leq \dfrac{1}{16}\left( M_{1}\left\vert b-a\right\vert
+M_{2}\left\vert d-c\right\vert \right)  \label{1}
\end{equation}%
\begin{eqnarray}
&&\left\vert \dfrac{f\left( a,c\right) +f\left( a,d\right) +f\left(
b,c\right) +f\left( b,d\right) }{4}-\dfrac{1}{\left( b-a\right) \left(
d-c\right) }\dint_{a}^{b}\dint_{c}^{d}f\left( x,y\right) dydx\right\vert 
\notag \\
&&  \notag \\
\ \ \ \ \ \ &\leq &\dfrac{1}{12}\left( M_{1}\left\vert b-a\right\vert
+M_{2}\left\vert d-c\right\vert \right)  \label{11}
\end{eqnarray}%
where $M_{1}=\left[ L_{1}+L_{3}+L_{5}+L_{7}\right] $ and $M_{2}=\left[
L_{2}+L_{4}+L_{6}+L_{8}\right] .$
\end{theorem}

\begin{proof}
Let $t,s\in \left[ 0,1\right] .$ Since $ts+s(1-t)+t(1-s)+(1-t)(1-s)=1,$ then
we have%
\begin{eqnarray}
&&\left\vert tsf(a,c)+s(1-t)f(b,c)+t(1-s)f(a,d)+(1-t)(1-s)f(b,d)\right. 
\notag \\
&&  \notag \\
&&\left. -f\left( ta+(1-t)b,sc+(1-s)d\right) \right\vert  \notag \\
&&  \notag \\
&=&\left\vert ts\left[ f(a,c)-f\left( ta+(1-t)b,sc+(1-s)d\right) \right]
\right.  \notag \\
&&  \label{2} \\
&&+s(1-t)\left[ f(b,c)-f\left( ta+(1-t)b,sc+(1-s)d\right) \right]  \notag \\
&&  \notag \\
&&+t(1-s)\left[ f(a,d)-f\left( ta+(1-t)b,sc+(1-s)d\right) \right]  \notag \\
&&  \notag \\
&&\left. +(1-t)(1-s)\left[ f(b,d)-f\left( ta+(1-t)b,sc+(1-s)d\right) \right]
\right\vert  \notag \\
&&  \notag \\
&\leq &ts\left[ (1-t)L_{1}\left\vert b-a\right\vert +(1-s)L_{2}\left\vert
d-c\right\vert \right] +s(1-t)\left[ tL_{3}\left\vert b-a\right\vert
+(1-s)L_{4}\left\vert d-c\right\vert \right]  \notag \\
&&  \notag \\
&&+t(1-s)\left[ (1-t)L_{5}\left\vert b-a\right\vert +sL_{6}\left\vert
d-c\right\vert \right] +(1-t)(1-s)\left[ tL_{7}\left\vert b-a\right\vert
+sL_{8}\left\vert d-c\right\vert \right]  \notag \\
&&  \notag \\
&=&\left( ts(1-t)\left[ L_{1}+L_{3}\right] +t(1-s)(1-t)\left[ L_{5}+L_{7}%
\right] \right) \left\vert b-a\right\vert  \notag \\
&&  \notag \\
&&+\left( ts(1-s)\left[ L_{2}+L_{6}\right] +s(1-s)(1-t)\left[ L_{4}+L_{8}%
\right] \right) \left\vert d-c\right\vert .  \notag
\end{eqnarray}%
If we choose $t=s=\frac{1}{2}$ in (\ref{2}), we get%
\begin{eqnarray}
&&\left\vert \dfrac{f\left( a,c\right) +f\left( a,d\right) +f\left(
b,c\right) +f\left( b,d\right) }{4}-f\left( \dfrac{a+b}{2},\dfrac{c+d}{2}%
\right) \right\vert  \notag \\
&&  \label{3} \\
&\leq &\frac{1}{8}\left( \left[ L_{1}+L_{3}+L_{5}+L_{7}\right] \left\vert
b-a\right\vert +\left[ L_{2}+L_{6}+L_{4}+L_{8}\right] \left\vert
d-c\right\vert \right) .  \notag
\end{eqnarray}%
Thus, if we put $ta+(1-t)b$ instead of $a,\ (1-t)a+tb$ instead of $b,\
sc+(1-s)d$ instead of $c$ and $(1-s)c+sd$ instead of $d$ in (\ref{3}),
respectively, then it follows that%
\begin{eqnarray}
&&\left\vert \dfrac{f\left( ta+(1-t)b,sc+(1-s)d\right) +f\left(
ta+(1-t)b,(1-s)c+sd\right) }{4}\right.  \notag \\
&&  \notag \\
&&+\dfrac{f\left( (1-t)a+tb,sc+(1-s)d\right) +f\left(
(1-t)a+tb,(1-s)c+sd\right) }{4}  \notag \\
&&  \label{4} \\
&&\left. -f\left( \dfrac{a+b}{2},\dfrac{c+d}{2}\right) \right\vert  \notag \\
&&  \notag \\
&\leq &\frac{1}{8}\left( \left[ L_{1}+L_{3}+L_{5}+L_{7}\right] \left\vert
1-2t\right\vert \left\vert b-a\right\vert +\left[ L_{2}+L_{6}+L_{4}+L_{8}%
\right] \left\vert 1-2s\right\vert \left\vert d-c\right\vert \right)  \notag
\end{eqnarray}%
for all $t,s\in \left[ 0,1\right] $. If we integrate the inequality (\ref{4}%
) with respect to $s,t$ on $\left[ 0,1\right] \times \left[ 0,1\right] $ 
\begin{eqnarray*}
&&\left\vert \frac{1}{4}\dint_{0}^{1}\dint_{0}^{1}\left[ f\left(
ta+(1-t)b,sc+(1-s)d\right) +f\left( ta+(1-t)b,(1-s)c+sd\right) \right]
dsdt\right. \\
&& \\
&&+\frac{1}{4}\dint_{0}^{1}\dint_{0}^{1}\left[ f\left(
(1-t)a+tb,sc+(1-s)d\right) +f\left( (1-t)a+tb,(1-s)c+sd\right) \right] dsdt
\\
&& \\
&&\left. -f\left( \dfrac{a+b}{2},\dfrac{c+d}{2}\right) \right\vert \\
&& \\
&\leq &\frac{1}{8}\left\{ \left[ L_{1}+L_{3}+L_{5}+L_{7}\right] \left\vert
b-a\right\vert \dint_{0}^{1}\dint_{0}^{1}\left\vert 1-2t\right\vert
dsdt\right. \\
&& \\
&&\left. +\left[ L_{2}+L_{6}+L_{4}+L_{8}\right] \left\vert d-c\right\vert
\dint_{0}^{1}\dint_{0}^{1}\left\vert 1-2s\right\vert dsdt\right\} .
\end{eqnarray*}%
Thus, using the change of the variable $x=ta+(1-t)b$, $y=(1-t)a+tb,\
u=sc+(1-s)d$ and $w=(1-s)c+sd$ for $t,s\in \lbrack 0,1]$, and 
\begin{equation*}
\dint_{0}^{1}\dint_{0}^{1}\left\vert 1-2t\right\vert
dsdt=\dint_{0}^{1}\dint_{0}^{1}\left\vert 1-2s\right\vert dsdt=\frac{1}{2}
\end{equation*}%
we obtain the inequality (\ref{1}).

Note that, by the inequality (\ref{2}), we write%
\begin{eqnarray}
&&\left\vert tsf(a,c)+s(1-t)f(b,c)+t(1-s)f(a,d)+(1-t)(1-s)f(b,d)\right. 
\notag \\
&&  \notag \\
&&\left. -f\left( ta+(1-t)b,sc+(1-s)d\right) \right\vert  \notag \\
&&  \label{5} \\
&\leq &\left( ts(1-t)\left[ L_{1}+L_{3}\right] +t(1-s)(1-t)\left[ L_{5}+L_{7}%
\right] \right) \left\vert b-a\right\vert  \notag \\
&&  \notag \\
&&+\left( ts(1-s)\left[ L_{2}+L_{6}\right] +s(1-s)(1-t)\left[ L_{4}+L_{8}%
\right] \right) \left\vert d-c\right\vert .  \notag
\end{eqnarray}%
for all $t,s\in \left[ 0,1\right] $. If we integrate the inequality (\ref{5}%
) with respect to $s,t$ on $\left[ 0,1\right] \times \left[ 0,1\right] ,$ we
have%
\begin{eqnarray*}
&&\left\vert \dfrac{f\left( a,c\right) +f\left( a,d\right) +f\left(
b,c\right) +f\left( b,d\right) }{4}-\dfrac{1}{\left( b-a\right) \left(
d-c\right) }\dint_{a}^{b}\dint_{c}^{d}f\left( x,y\right) dydx\right\vert \\
&& \\
&\leq &\frac{1}{12}\left( \left[ L_{1}+L_{3}+L_{5}+L_{7}\right] \left\vert
b-a\right\vert +\left[ L_{2}+L_{6}+L_{4}+L_{8}\right] \left\vert
d-c\right\vert \right)
\end{eqnarray*}%
and so we have the inequality (\ref{11}), where we use the fact that%
\begin{equation*}
\dint\limits_{0}^{1}\dint\limits_{0}^{1}st(1-t)dsdt=\dint\limits_{0}^{1}%
\dint\limits_{0}^{1}s(1-s)(1-t)dsdt=\frac{1}{12}.
\end{equation*}%
This completes the proof.
\end{proof}

\section{The Mapping $H$}

For a $L$-Lipschitzian function $f:\Delta \subset \mathbb{R}^{2}\rightarrow 
\mathbb{R}$, we can define a mapping $H:\left[ 0,1\right] \times \left[ 0,1%
\right] \rightarrow \mathbb{R}$ by%
\begin{equation*}
H(t,s):=\dfrac{1}{\left( b-a\right) \left( d-c\right) }\dint_{a}^{b}%
\dint_{c}^{d}f\left( tx+(1-t)\frac{a+b}{2},sy+(1-s)\frac{c+d}{2}\right) dydx.
\end{equation*}%
Now, we give some properties of this mapping as follows:

\begin{theorem}
\label{s} Suppose that $f:\Delta \subset \mathbb{R}^{2}\rightarrow \mathbb{R}
$ be $L$-Lipschitzian on $\Delta :=\left[ a,b\right] \times \left[ c,d\right]
$. Then:
\end{theorem}

\textit{(i) The mapping }$H$\textit{\ is }$L$\textit{-Lipschitzian on} $%
\left[ 0,1\right] \times \left[ 0,1\right] .$

\textit{(ii) We have the following inequalities}%
\begin{equation}
\left\vert H(t,s)-f\left( \dfrac{a+b}{2},\dfrac{c+d}{2}\right) \right\vert
\leq \frac{L_{1}t}{4}\left( b-a\right) +\frac{L_{2}s}{4}\left( d-c\right)
\label{6}
\end{equation}%
\begin{equation}
\left\vert H(t,s)-\dfrac{1}{\left( b-a\right) \left( d-c\right) }%
\dint_{a}^{b}\dint_{c}^{d}f\left( x,y\right) dydx\right\vert \leq \frac{%
L_{1}(1-t)}{4}\left( b-a\right) +\frac{L_{2}(1-s)}{4}\left( d-c\right) .
\label{7}
\end{equation}

\begin{proof}
\textit{(i)} Let $t_{1},t_{2},s_{1},s_{2}\in \left[ 0,1\right] .$ Then, we
have%
\begin{eqnarray*}
&&\left\vert H(t_{2},s_{2})-H(t_{1},s_{1})\right\vert \\
&& \\
&=&\dfrac{1}{\left( b-a\right) \left( d-c\right) }\left\vert
\dint_{a}^{b}\dint_{c}^{d}f\left( t_{2}x+(1-t_{2})\frac{a+b}{2}%
,s_{2}y+(1-s_{2})\frac{c+d}{2}\right) dydx\right. \\
&& \\
&&\left. -\dint_{a}^{b}\dint_{c}^{d}f\left( t_{1}x+(1-t_{1})\frac{a+b}{2}%
,s_{1}y+(1-s_{1})\frac{c+d}{2}\right) dydx\right\vert \\
&& \\
&\leq &\dfrac{1}{\left( b-a\right) \left( d-c\right) }\dint_{a}^{b}%
\dint_{c}^{d}\left\vert f\left( t_{2}x+(1-t_{2})\frac{a+b}{2}%
,s_{2}y+(1-s_{2})\frac{c+d}{2}\right) \right. \\
&& \\
&&\left. -f\left( t_{1}x+(1-t_{1})\frac{a+b}{2},s_{1}y+(1-s_{1})\frac{c+d}{2}%
\right) dydx\right\vert \\
&& \\
&=&\dfrac{1}{\left( b-a\right) \left( d-c\right) }\dint_{a}^{b}\dint_{c}^{d}%
\left[ L_{1}\left\vert t_{2}-t_{1}\right\vert \left\vert x-\frac{a+b}{2}%
\right\vert +L_{2}\left\vert s_{2}-s_{1}\right\vert \left\vert y-\frac{c+d}{2%
}\right\vert \right] dydx \\
&& \\
&=&\frac{L_{1}\left( b-a\right) }{4}\left\vert t_{2}-t_{1}\right\vert +\frac{%
L_{2}\left( d-c\right) }{4}\left\vert s_{2}-s_{1}\right\vert ,
\end{eqnarray*}%
i.e., for all $t_{1},t_{2},s_{1},s_{2}\in \left[ 0,1\right] ,$%
\begin{equation}
\left\vert H(t_{2},s_{2})-H(t_{1},s_{1})\right\vert \leq \frac{L_{1}\left(
b-a\right) }{4}\left\vert t_{2}-t_{1}\right\vert +\frac{L_{2}\left(
d-c\right) }{4}\left\vert s_{2}-s_{1}\right\vert ,  \label{9}
\end{equation}%
which yields that the mapping $H$ is $L$-Lipschitzian on $\left[ 0,1\right]
\times \left[ 0,1\right] .$

\textit{(ii)} The inequalities (\ref{6}) and (\ref{7}) follow from (\ref{9})
by choosing $t_{1}=0,\ t_{2}=t,\ s_{1}=0,\ s_{2}=s$ and $t_{1}=1,\ t_{2}=t,\
s_{1}=1,\ s_{2}=s,$ respectively.
\end{proof}

Another result which is connected in a sense with the inequality (\ref{11})
is also given in the following:

\begin{theorem}
Under the assumptions Theorem \ref{s}, then we get the following inequality%
\begin{eqnarray}
&&\left\vert \dfrac{f\left( at+(1-t)\frac{a+b}{2},cs+(1-s)\frac{c+d}{2}%
\right) +f\left( at+(1-t)\frac{a+b}{2},ds+(1-s)\frac{c+d}{2}\right) }{4}%
\right.  \notag \\
&&  \notag \\
&&+\dfrac{f\left( bt+(1-t)\frac{a+b}{2},cs+(1-s)\frac{c+d}{2}\right)
+f\left( bt+(1-t)\frac{a+b}{2},ds+(1-s)\frac{c+d}{2}\right) }{4}  \notag \\
&&  \label{10} \\
&&\left. -\frac{1}{(n_{2}-n_{1})(m_{2}-m_{1})}\dint_{n_{1}}^{n_{2}}%
\dint_{m_{1}}^{m_{2}}f(u,w)dwdu\right\vert  \notag \\
&&  \notag \\
&\leq &\dfrac{1}{12}\left( M_{1}\left\vert n_{2}-n_{1}\right\vert
t+M_{2}\left\vert m_{2}-m_{1}\right\vert s\right)  \notag
\end{eqnarray}%
where $M_{1}=\left[ L_{1}+L_{3}+L_{5}+L_{7}\right] $ and $M_{2}=\left[
L_{2}+L_{4}+L_{6}+L_{8}\right] .$
\end{theorem}

\begin{proof}
If we denote $n_{1}=at+(1-t)\frac{a+b}{2},\ n_{2}=bt+(1-t)\frac{a+b}{2},\
m_{1}=cs+(1-s)\frac{c+d}{2}$ and $m_{2}=ds+(1-s)\frac{c+d}{2},$ then, we have%
\begin{equation*}
H(t,s)=\frac{1}{(n_{2}-n_{1})(m_{2}-m_{1})}\dint_{n_{1}}^{n_{2}}%
\dint_{m_{1}}^{m_{2}}f(u,w)dwdu.
\end{equation*}%
Now, using the inequality (\ref{11}) applied for $n_{1},n_{2},m_{1}$ and $%
m_{2},$ we have%
\begin{eqnarray*}
&&\left\vert \dfrac{f\left( n_{1},m_{1}\right) +f\left( n_{1},m_{2}\right)
+f\left( n_{2},m_{1}\right) +f\left( n_{2},m_{2}\right) }{4}\right. \\
&& \\
&&\left. -\frac{1}{(n_{2}-n_{1})(m_{2}-m_{1})}\dint_{n_{1}}^{n_{2}}%
\dint_{m_{1}}^{m_{2}}f(u,w)dwdu\right\vert \\
&& \\
&\leq &\dfrac{1}{12}\left( M_{1}\left\vert n_{2}-n_{1}\right\vert
+M_{2}\left\vert m_{2}-m_{1}\right\vert \right)
\end{eqnarray*}%
from which we have the inequality (\ref{10}). This completes the proof.
\end{proof}

\end{document}